\documentclass{amsart}
\usepackage{}
\usepackage{color}
\usepackage{amssymb}
\usepackage{amsthm}

\newcommand{\tF}{\tilde{F}}

\newtheorem{theorem}{Theorem}[section]
\newtheorem{lemma}[theorem]{Lemma}

 \theoremstyle{definition}
\newtheorem{definition}[theorem]{Definition}

\theoremstyle{remark}
\newtheorem{remark}[theorem]{Remark}

\numberwithin{equation}{section}

\begin{document}

\title[curvature estimates]
{Curvature estimates for $p$-convex hypersurfaces of prescribed curvature}

\author{Weisong Dong}
\address{School of Mathematics, Tianjin University, 135 Yaguan Road,
Jinnan, Tianjin, China, 300354}
\email{dr.dong@tju.edu.cn}


\begin{abstract}

In this paper, we establish curvature estimates for $p$-convex hypersurfaces in $\mathbb{R}^{n+1}$ of prescribed curvature
with $p \geq \frac{n}{2}$.
The existence of a star-shaped hypersurface of prescribed curvature is obtained.
We also prove a type of interior $C^2$ estimates for solutions to the Dirichlet problem of the corresponding equation.



\end{abstract}

\maketitle

\section{Introduction}

Let $M\subset \mathbb{R}^{n+1}$ be a closed hypersurface and let
$\kappa (X) = (\kappa_1, \cdots, \kappa_n)$ be the principal curvatures of $M$ at $X$.
Given $1\leq p \leq n$, a $C^2$ regular hypersurface $M$ is called \emph{$p$-convex} if, at each $X\in M$,
$\kappa(X)$ satisfies
\[
\kappa_{i_1} + \cdots + \kappa_{i_p} \geq 0, \: \forall \;  1\leq i_1 < \cdots < i_p\leq n.
\]
In other words, the sum of the $p$ smallest principal curvatures is nonnegative at each point of $M$.
The notion of $p$-convexity goes back to Wu \cite{Wu} and has been studied extensively by Wu \cite{Wu}, Sha \cite{Sha1,Sha2} and Harvey-Lawson \cite{HL12,HL13}.

In this paper, we are interested in finding a $p$-convex hypersurface $M\subset \mathbb{R}^{n+1}$ of prescribed curvature as below
\begin{equation}
\label{eqn}
\Pi_{1\leq i_1 < \cdots < i_p\leq n} (\kappa_{i_1} + \cdots + \kappa_{i_p}) = f(X, \nu(X)),\; \forall\;  X \in M,
\end{equation}
where $\nu(X)$ is the unit outer normal of $M$ at $X$, the function $f (X, \nu) \in C^2 (\Gamma)$ is positive and
$\Gamma$ is an open neighborhood of unit normal bundle of $M$ in $\mathbb{R}^{n+1} \times \mathbb{S}^n$.
The Gaussian curvature equation, that corresponds to $p=1$ in \eqref{eqn}, was studied by Oliker \cite{Oliker}.
The mean curvature equation corresponding to $p=n$ in \eqref{eqn} was studied by Bakelman-Kantor \cite{BK} and Treibergs-Wei \cite{TreW}.
For general curvature equations, see Caffarelli-Nirenberg-Spruck \cite{CNS4} and Gerhardt \cite{Gerhardt}.
When $p=n-1$, the equation was studied by Chu-Jiao \cite{CJ} and,
in complex settings, it is related to
the Gauduchon conjecture which was
solved by Sz\'ekelyhidi-Tosatti-Weinkove \cite{S-T-W}.
For some previous work on this topic, see Tosatti-Weinkove \cite{TW1,TW2} and Fu-Wang-Wu \cite{FWW1,FWW2}.

It is of great interest in geometry and PDEs to derive a $C^2$ estimate for equation \eqref{eqn} for general $f(X,\nu(X))$.
We have the following main result.
\begin{theorem}
\label{thm1}
Suppose $M\subset \mathbb{R}^{n+1}$ is a closed star-shaped $p$-convex hypersurface
with $p  \geq \frac{n}{2}$ satisfying the curvature equation \eqref{eqn}.
Then, there is a positive constant
$C$ such that
\begin{equation}
\label{esti}
\sup_{X\in M, i = 1,\cdots, n}|\kappa_i(X)| \leq C ,
\end{equation}
where $C$ depends on $n$, $p$, $|M|_{C^1}$, $\inf f$ and $|f|_{C^2}$.
\end{theorem}


We remark that in the theorem only a few conditions are assumed on $f$.
Usually, to derive $C^2$ estimate for elliptic equations which are not strictly elliptic there should be some extra assumptions on $f$
due to the dependency on $\nu(X)$.
We refer the reader to Ivochkina \cite{I1,I2}, Guan-Guan \cite{GG}, Guan-Lin-Ma \cite{GLM}, Guan-Li-Li \cite{GLL} and Guan-Jiao \cite{GJ}
for more details.
Moreover, Guan-Ren-Wang in \cite{GRW} showed that
for the following curvature equation
\[
\frac{\sigma_k}{\sigma_l} (\kappa (X)) = f(X, \nu(X)),\; \forall X \in M,
\]
where $0 < l < k \leq n$ and $\sigma_k$ is the $k$-th elementary symmetric function,
estimate \eqref{esti} fails generally,
though it may hold for special $f$ as in Guan-Guan \cite{GG}.
When $l=0$, some results are known for general $f$. For instance,
estimate \eqref{esti} was proved for $k=n$ by Caffarelli-Nirenberg-Spruck \cite{CNS1}
 and, for $2 \leq k \leq n$, Guan-Ren-Wang \cite{GRW} obtained the estimate for convex solutions.
For $k= n-1$ and $n-2$, estimate \eqref{esti} was established by Ren-Wang \cite{RW1,RW2}.
They also conjectured that the estimate still holds for $k> \frac{n}{2}$ in \cite{RW2,RW3}.
When $f$ is independent of $\nu$, Caffarelli-Nirenberg-Spruck \cite{CNS4} proved
the $C^2$ estimate for a general class of fully nonlinear curvature equations.
For hypersurfaces of prescribed curvature in Riemannian manifolds and Minkowski space, see \cite{CLW} and \cite{WX,RWX}.
We also refer the reader to Guan-Zhang \cite{GZ} and references therein for a class of curvature equations arising
from convex geometry.

To obtain the existence of a $p$-convex hypersurface satisfying the prescribed curvature equation \eqref{eqn},
we assume the following two conditions on $f$.
The first one is that there exists two positive constants $r_1 < 1 < r_2$ such that
\begin{equation}
\label{f1}
\begin{aligned}
f\Big(X, \frac{X}{|X|}\Big) \geq &\; \frac{p^{C_n^p}}{r_1^{C_n^p}}, \; \mbox{for}\; |X| = r_1;\\
f\Big(X, \frac{X}{|X|}\Big) \leq &\; \frac{p^{C_n^p}}{r_2^{C_n^p}}, \; \mbox{for}\; |X| = r_2.
\end{aligned}
\end{equation}
This condition is used to derive $C^0$ estimates.
The second one is that for any fixed unit vector $\nu$,
\begin{equation}
\label{f2}
\frac{\partial}{\partial \rho} \Big(\rho^{ C_n^p } f (X, \nu)\Big) \leq 0, \; \mbox{where}\; \rho = |X|,
\end{equation}
and will be used to derive $C^1$ estimates. Actually, with suitable assumptions of $f$,
Li \cite{L} proved that the interior gradient estimate holds.

By the continuity method argument as in \cite{CNS4}, we can obtain the following
result.
\begin{theorem}
\label{thm2}
Let $f\in C^2 ( (\overline{B_{r_2}} \backslash B_{r_1} )\times \mathbb{S}^n )$ be a positive function satisfying \eqref{f1} and \eqref{f2}.
Then equation \eqref{eqn} has a unique $C^{3, \alpha}$ star-shaped $p$-convex solution $M$
in $\{X \in \mathbb{R}^{n+1} : r_1 \leq |X| \leq r_2\}$ for any $\alpha \in (0, 1)$
as long as $p \geq \frac{n}{2}$.
\end{theorem}

The method of proving Theorem \ref{thm1} can be applied to obtain
an interior $C^2$ estimate for the Dirichlet problem of the corresponding equation in the Euclidean space.
Suppose that $\Omega$ is a bounded domain in $\mathbb{R}^n$.
For a function $u \in C^2(\Omega)$, denote by $\lambda (D^2 u) = (\lambda_1, \cdots, \lambda_n)$ the eigenvalues of the Hessian $D^2 u$.
We say that $u\in C^2 (\Omega)$ is \emph{$p$-plurisubharmonic} if the eigenvalues of $D^2 u$ satisfy
$\lambda_{i_1} + \cdots + \lambda_{i_p} \geq 0$, for all $1\leq i_1 < \cdots < i_p\leq n$ (see \cite{HL13}).
Given a $C^2$ $p$-plurisubharmonic function $v$ on $\bar \Omega$,
consider the following Dirichlet problem,
\begin{equation}
\label{eqn-Dirichlet}
\Pi_{1\leq i_1 < \cdots < i_p\leq n} (\lambda_{i_1} + \cdots + \lambda_{i_p}) = f(x,u, D u), \;\mbox{in}\; \Omega
\end{equation}
with boundary data
\[
 u = v, \;\mbox{on}\; \partial \Omega,
\]
where $f\in C^2(\bar \Omega \times \mathbb{R} \times \mathbb{R}^n)$ is a positive function.
By the same argument as that of Theorem \ref{thm1},
we can prove the following interior estimate.
\begin{theorem}
\label{thm3}
Suppose that a $p$-plurisubharmonic function $u\in C^4(\Omega)\cap C^{1,1} (\bar \Omega)$ is a solution to the Dirichlet problem of equation \eqref{eqn-Dirichlet} and satisfies $ u< v $ in $\Omega$.
Then, there exist constants $C$ and $\beta$ depending only on $n,p$, $|u|_{C^1}$, $|v|_{C^1}$,
$\inf f$, $|f|_{C^2}$ and $\Omega$ such that
\begin{equation}
\label{esti-2}
\sup_{\Omega} (v-u)^\beta \Delta u \leq C
\end{equation}
as long as $p \geq \frac{n}{2}$.
\end{theorem}
\begin{remark}
As a byproduct of the proof of the above theorem, one can conclude the following global $C^2$ estimate for equation \eqref{eqn-Dirichlet}
\[
\sup_{\Omega} |D^2 u| \leq C (1 + \sup_{\partial \Omega} |D^2 u|),
\]
where $C$ is a constant as in Theorem \ref{thm3}.
\end{remark}

We shall only give an outline for the proof of Theorem \ref{thm3}, as it is almost the same as that of Theorem \ref{thm1}.
The estimate \eqref{esti-2} can also be seen in some sense as a generalization of Theorem 0.4 in \cite{Dinew}, since there the right hand side function $f$ does not depend on $Du$.
Such an estimate for the $k$-Hessian equation $\sigma_k (\lambda) = f(x, u)$ has been proved by Chou-Wang \cite{CW}.
The function $f$ depending on $\nu$ or $Du$ creates substantial difficulties to derive a $C^2$ estimate,
as the bad term $-Ch_{11}$ or $-Cu_{11}$ appears when one applies the maximum principle to the test function.
One way to overcome this is like in \cite{GRW, RW1, RW2} to control the bad third order terms,
which is very hard even for the $k$-Hessian equation.
Another way is as in \cite{CJ} to control $-Ch_{11}$ firstly by good terms (see Lemma \ref{lem2}),
and then the bad third order terms can be eliminated easily by Lemma \ref{lem4}.
Thanks to Dinew \cite{Dinew}, where many properties of the operator have been proved, we can follow the argument
in \cite{CJ}
to prove the estimate.

The paper is organized as follows. In Section 2, we recall some properties of the operator from \cite{Dinew}.
In Section 3, we prove the curvature estimate. In Section 4 we derive the gradient estimate and in Section 5 we apply the continuity method
to prove Theorem \ref{thm2}.
Finally, we give an outline of the proof of Theorem \ref{thm3} in Section 6.

\textbf{Acknowledgements}:
We would like to thank the anonymous referees for
helpful comments
and
especially for pointing out a mistake.
This work was partially supported by the National Natural Science Foundation of China, No.
11801405.

\section{Preliminaries}

Let $\mathbb{S}^n$ be the unit sphere in $\mathbb{R}^{n+1}$ and let $\nabla$ be the connection on it.
Assume that $M$ is star-shaped with respect to the origin, i.e. the position vector $X$ of $M$
can be written as $X(x) = \rho(x) x$, where $x\in \mathbb{S}^n$.
Then the unit outer normal of $M$ is given by
\[
\nu = \frac{\rho x - \nabla \rho}{\sqrt{\rho^2 + |\nabla \rho|^2}}.
\]
Let $\{e_1, \cdots, e_n\}$ be a smooth local orthonormal frame on $\mathbb{S}^n$.
Then the metric of $M$ is given by
$g_{ij} = \rho^2 \delta_{ij} + \rho_i \rho_j$,
and the second fundamental form of $M$ is
\begin{equation}
\label{secondForm}
h_{ij} = \frac{\rho^2 \delta_{ij} + 2 \rho_i \rho_j - \rho \rho_{ij}}{\sqrt{\rho^2 + |\nabla \rho|^2} }.
\end{equation}
The principal curvatures $\kappa = (\kappa_1, \cdots, \kappa_n)$ are the eigenvalues of $h_{ij}$
with respect to $g_{ij}$.

At a point $X_0$ in $M$, choose a local orthonormal frame
$\{e_1, e_2 \cdots, e_n\}$.
The following geometric formulas are well known:
\begin{equation}
\label{formula1}
\begin{aligned}
X_{ij} =&\; -h_{ij}\nu \;\;(\mbox {Gauss formula}),\\
(\nu)_i =&\; h_{ij} e_j \;\; (\mbox{Weingarten equation}),\\
h_{ijk} =&\; h_{ikj}\;\;(\mbox{Codazzi formula}),\\
R_{ijkl} =&\; h_{ik}h_{jl} -h_{il}h_{jk}\;\;(\mbox{Gauss equation}),
\end{aligned}
\end{equation}
where $R_{ijkl}$ is the $(4,0)$-Riemannian curvature tensor and the formula
\begin{equation}
\label{formula2}
h_{ijkl} =h_{klij} + h_{mk} (h_{mj} h_{il} - h_{ml}h_{ij}) + h_{mi} (h_{mj} h_{kl} - h_{ml} h_{kj}).
\end{equation}

We recall the $p$-convex cones introduced by Harvey and Lawson \cite{HL13}.
\begin{definition}
\label{def}
Let $p\in \{1, \cdots, n\}$. The cone $\mathcal{P}_p$ is defined by
\[
\mathcal{P}_p = \{ (\lambda_1, \cdots, \lambda_n) \in \mathbb{R}^n | \; \forall \; 1\leq i_1 < i_2 < \cdots < i_p\leq n,\;
\lambda_{i_1} + \cdots + \lambda_{i_p} > 0 \}.
\]
Associated to $\mathcal{P}_p$ is the cone of symmetric $n \times n$ matrices defined by
\[
P_p = \{A | \; \forall \; 1\leq i_1 < i_2 < \cdots < i_p\leq n,\;
\lambda_{i_1} (A) + \cdots + \lambda_{i_p} (A) > 0 \}.
\]
We call $A$ is $p$-positive if $A\in P_p$.
\end{definition}

For convenience, we introduce the following notations
\[
F (h_{ij}) := F (\kappa) = \Pi_{1\leq i_1 < \cdots < i_p\leq n} (\kappa_{i_1} + \cdots + \kappa_{i_p}) \; \mbox{and} \; \tilde F = F^{\frac{1}{C_n^p}},
\]
where $C_n^p = \frac{n!}{p! (n-p)!}$.
Equation \eqref{eqn} then can be written as
\begin{equation}
\label{eqn'}
\tilde F (h_{ij}) := \tilde F(\kappa) = \tilde f(X, \nu(X)),
\end{equation}
where $\kappa = (\kappa_1, \cdots, \kappa_n)$ and $\tilde f = f^{ \frac{1}{C_n^p} }$.
Denote
\[
F^{ij} = \frac{\partial F}{\partial h_{ij}}, \;
F^{ij,kl} = \frac{\partial^2 F}{\partial h_{ij} \partial h_{kl} },\;
\mbox{and} \; \mathcal{F} = \sum F^{ii}.
\]
Direct calculations show that
 \[
 \tilde F^{ij} = \frac{1}{C_n^p} F^{\frac{1}{C_n^p} - 1} F^{ij}
 \]
 and
 \[
 \tilde F^{ij,kl} = \frac{1}{C_n^p} F^{\frac{1}{C_n^p} - 1} F^{ij,kl}
 + \frac{1}{C_n^p} \Big( \frac{1}{C_n^p} - 1 \Big) F^{\frac{1}{C_n^p} - 2} F^{ij} F^{kl}.
 \]
We remark that $\tilde F$ is concave with respect to $h_{ij}$ by Lemma 1.13 and Corollary 1.14 in \cite{Dinew}.
And the equation is elliptic as the matrix$\{ \frac{\partial \tF}{\partial h_{ij}} \}$ is positive definite for $\{h_{ij}\} \in P_p$.

Now we do some basic calculations which will be used in the next section.
Our calculations are carried out at a point $X_0$ on the hypersurface $M$, and we use coordinates such that at this point
$\{ h_{ij} \}$ is diagonal and its eigenvalues with respect to $g_{ij}$ are ordered as $\kappa_1 \geq \kappa_2 \geq \cdots \geq \kappa_n$.
Note that $F^{ij} $ is also diagonal at $X_0$ and we have the following formulas
\[
F^{kk} = \frac{\partial F}{\partial \kappa_k} =
\sum_{k\in\{i_1, \cdots, i_p\} } \frac{F (\kappa)}{\kappa_{i_1} + \cdots + \kappa_{i_p} },
\]
for which we refer to Lemma 1.10 in \cite{Dinew}.
We also have formulas for the second order derivatives of $F$ at $X_0$ as below
\[
F^{kk,ll}= \frac{\partial^2 F}{\partial \kappa_k \partial \kappa_l} =
\sum_{ \substack{ k\in\{i_1, \cdots, i_p\} \\ l\in \{j_1, \cdots, j_p\} \\ \{i_1, \cdots, i_p\} \neq \{j_1, \cdots, j_p\}} }
\frac{F (\kappa)}{ (\kappa_{i_1} + \cdots + \kappa_{i_p}) (\kappa_{j_1} + \cdots + \kappa_{j_p}) },
\]
and, for $k\neq r$,
\[
F^{kr, r k} = \frac{F^{k k} - F^{r r} }{\kappa_k - \kappa_r}
=
- \sum_{ \substack{ k\notin\{i_1, \cdots, i_p\} \ni r  \\ r\notin \{j_1, \cdots, j_p\} \ni k
\\ \{i_1, \cdots, i_p\} \backslash \{r\} = \{j_1, \cdots, j_p\} \backslash \{k\}
}  }
\frac{ F(\kappa) }{( \kappa_{i_1} + \cdots + \kappa_{i_p} )(\kappa_{j_1} + \cdots + \kappa_{j_p})}.
\]
Otherwise, we have $F^{ij,kl} = 0$.
See Lemma 1.12 in \cite{Dinew} for the above formulas. These formulas can also be easily obtained from Theorem 5.5 in \cite{B}.
The following properties of the function $F$ which are very similar to the properties of $\sigma_k$ were proved by Dinew \cite{Dinew}.

\begin{lemma}[\cite{Dinew}]
\label{Dinew1}
Suppose that the diagonal matrix $A = diag(\lambda_1, \cdots, \lambda_n)$ belongs to $P_p$ and that $\lambda_1 \geq \cdots \geq \lambda_n$.
Then,
\begin{itemize}
\item [(1)] $\tF^{11} (A) \lambda_1 \geq \frac{1}{n} \tF (A);$
\item [(2)] $\sum_{k=1}^n \tF^{kk} (A) \geq p;$
\item [(3)] $\sum_{k=1}^n F^{kk} (A) \lambda_{k} = C_n^p F(A);$
\item [(4)] there is a constant $\theta = \theta(n, p)$ such that $F^{jj} (A) \geq \theta \sum F^{ii}$ for all $j \geq n-p+1.$
\end{itemize}
\end{lemma}
For the reader's convenience, we provide a short proof of the above lemma in the appendix.



\section{Curvature Estimates}

Set $u = \langle X, \nu \rangle$, which is the support function of the hypersurface $M$.
Clearly, we have
\[
u = \frac{\rho^2}{ \sqrt{\rho^2 + |\nabla \rho|^2} }.
\]
There exists a positive constant $C$ depending on $\inf_M \rho$ and $|\rho|_{C^1}$ such that
\[
\frac{1}{C} \leq \inf_M u \leq u \leq \sup_M u \leq C.
\]

In order to prove Theorem \ref{thm1}, we consider the following auxiliary function
\[
G =  \log \kappa_{max } - \log (u-a) + \frac{A}{2} |X|^2
\]
where $\kappa_{max}$ is the largest principal curvature, $a = \frac{1}{2} \inf_M u > 0$ and $A \geq 1$ is a large constant to be determined.
Suppose the maximum of $G$ is achieved at a point $X_0\in M$.
Choose a local orthonormal frame $\{e_1, \cdots, e_n\}$ around $X_0$ such that
\[
h_{ij} = \delta_{ij} h_{ii} \;\mbox{and}\; h_{11} \geq h_{22} \geq \cdots \geq h_{nn} \; \mbox{at}\;  X_0.
\]
Since $\kappa_{max}$ may not be differentiable, we define a new function $\hat G$ near $X_0$ by
\[
\hat G =  \log h_{11} - \log (u-a) + \frac{A}{2} |X|^2.
\]
It is easy to see $\hat G$ achieves a maximum at $X_0$.
Now, differentiating $\hat G$ at $X_0$ twice yields that
\begin{equation}
\label{diff1}
0 = \frac{ h_{11i}}{h_{11}} - \frac{u_i  }{u-a} + A \langle X, e_i\rangle
\end{equation}
and
\begin{equation}
\label{diff2}
0\geq \frac{h_{11ii}}{h_{11}} -\Big(\frac{h_{11i}}{h_{11}}\Big)^2 - \frac{ u_{ii}  }{u-a} + \Big(\frac{ u_i }{u-a}\Big)^2
+ A (1 + \langle X, X_{ii}\rangle).
\end{equation}
Contracting \eqref{diff2} with $\tF^{ii}$, we get
\begin{equation}
\label{S2-1}
0\geq \frac{\tF^{ii} h_{11ii}}{h_{11}} - \frac{\tF^{ii}h_{11i}^2}{h_{11}^2} - \frac{ \tF^{ii} u_{ii}  }{u-a} + \frac{ \tF^{ii}u_i^2 }{(u-a)^2}
+ A \tF^{ii} (1 + \langle X, X_{ii}\rangle).
\end{equation}

\begin{lemma}
\label{lem1}
We have
\begin{equation}
\label{L1}
\begin{aligned}
0 \geq &\; -\frac{2}{h_{11}} \sum_{i\geq 2} \tF^{1i,i1} h_{11i}^2 - \frac{\tF^{ii}h_{11i}^2}{h_{11}^2} - Ch_{11}\\
&\; + \frac{a \tF^{ii}h_{ii}^2 }{u-a} + \frac{\tF^{ii} u_i^2}{(u-a)^2} + A \sum \tF^{ii} - CA.
\end{aligned}
\end{equation}
\end{lemma}

\begin{proof}

From the formula \eqref{formula2}, we have
\begin{equation}
\label{S2-3}
\begin{aligned}
\tF^{ii}h_{11ii} = \tF^{ii}h_{ii}h_{11}^2 - \tF^{ii}h_{ii}^2h_{11} + \tF^{ii}h_{ii11}.
\end{aligned}
\end{equation}
Differentiating equation \eqref{eqn} twice at $X_0$, we obtain
\begin{equation}
\label{deqn1}
\tilde F^{ii} h_{iik} = (d_X \tilde f) (e_k) + h_{kk} ( d_\nu \tilde f ) (e_k)
\end{equation}
and
\begin{equation}
 \label{equa2}
\begin{aligned}
\tilde F^{ii} h_{iikk} + \tilde F^{pq,rs} h_{pqk} h_{rsk}
\geq -C - C h_{11}^2 + \sum_l h_{lkk} (d_\nu \tilde f) (e_l).
  \end{aligned}
\end{equation}
By the concavity of $\tF$ and the Codazzi formula, we have
\begin{equation}
\label{Fconcave}
- \tilde F^{pq,rs} h_{pq1} h_{rs1} \geq - 2 \sum_{i\geq 2} \tF^{1i,i1} h_{11i}^2.
\end{equation}
Note that by Lemma \ref{Dinew1} (3) we have $\tF^{ii} h_{ii} = \tilde f$.
Hence, we see that
\begin{equation}
\label{Fh}
\tF^{ii} h_{11ii} \geq - 2 \sum_{i\geq 2} \tF^{1i,i1} h_{11i}^2 - \tF^{ii}h_{ii}^2h_{11} + \sum_l h_{l11} (d_\nu \tilde f) (e_l) -C - C h_{11}^2.
\end{equation}

We now compute the term $\tF^{ii}u_{ii}$. By \eqref{formula1}, we have
\[
u_{i} = h_{ii}\langle X, e_i\rangle \; \mbox{and}\; u_{ii} = \sum_k h_{iik} \langle X, e_k\rangle - u h_{ii}^2 + h_{ii}.
\]
Hence, we obtain
\begin{equation}
\label{Fu}
\begin{aligned}
\tF^{ii} u_{ii} = &\; \sum_k \tF^{ii}h_{iik} \langle X, e_k\rangle - u \tF^{ii}h_{ii}^2 + \tilde f\\
\leq &\; \sum_k h_{kk} (d_\nu \tilde f) (e_k)\langle X, e_k\rangle - u \tF^{ii}h_{ii}^2 + C.
\end{aligned}
\end{equation}

By the Gauss formula, we have
\begin{equation}
\label{FX}
\langle X, X_{ii}\rangle = -h_{ii} \langle X, \nu \rangle = -h_{ii} u.
\end{equation}
Substituting \eqref{Fh}, \eqref{Fu} and \eqref{FX} in \eqref{S2-1}, we obtain that
\begin{equation}
\label{S2-2}
\begin{aligned}
0 \geq &\,   - \frac{2}{h_{11}}  \sum_{i\geq 2} \tF^{1i,i1} h_{11i}^2 - \tF^{ii}h_{ii}^2
+ \frac{1}{h_{11}} \sum_l h_{l 11} (d_\nu \tilde f) (e_l) - C h_{11}\\
&\; - \frac{\tF^{ii}h_{11i}^2}{h_{11}^2} - \frac{ \sum_k h_{kk} (d_\nu \tilde f) (e_k)\langle X, e_k\rangle  }{u-a}
+ \frac{ u \tF^{ii}h_{ii}^2 }{u-a} - \frac{C}{u-a}\\
&\; + \frac{ \tF^{ii}u_i^2 }{(u-a)^2}
+ A \sum \tF^{ii} - CA.
\end{aligned}
\end{equation}
By the Codazzi formula, $u_k = h_{kk}\langle X, e_k\rangle$ and \eqref{diff1}, we have
\[
\frac{1}{h_{11}} \sum_k h_{k11} (d_\nu \tilde f) (e_k) - \frac{ h_{kk} (d_\nu \tilde f) (e_k)\langle X, e_k\rangle  }{u-a} \geq -CA.
\]
Therefore, we arrive at
\begin{equation}
\label{S2-4}
\begin{aligned}
0 \geq &\,   - \frac{2}{h_{11}}  \sum_{i\geq 2} \tF^{1i,i1} h_{11i}^2 - \tF^{ii}h_{ii}^2 - C h_{11} - \frac{C}{u-a}\\
&\; - \frac{\tF^{ii}h_{11i}^2}{h_{11}^2}
+ \frac{ u \tF^{ii}h_{ii}^2 }{u-a} + \frac{ \tF^{ii}u_i^2 }{(u-a)^2}
+ A \sum \tF^{ii} - CA,
\end{aligned}
\end{equation}
which is just the inequality \eqref{L1}.

\end{proof}

Next, we deal with the bad term $- C h_{11}$.
\begin{lemma}
\label{lem2}
Suppose $p\geq \frac{n}{2}$.
If $h_{11}$ is large enough, we have
\begin{equation}
\label{L2}
Ch_{11} \leq  \frac{ a \tF^{ii}h_{ii}^2 }{2 (u-a)} + \frac{A}{2} \sum \tF^{ii}
\end{equation}
for sufficiently large $A$.
\end{lemma}

\begin{proof}
Note that
\[
\kappa_{n-p+1} + \kappa_{n-p+2} + \cdots + \kappa_n > 0.
\]
We divide the proof into two cases.

Case 1. Suppose $\kappa_n \leq - \delta \kappa_1$,
where $\delta > 0$ is a small constant to be determined later.
 By Lemma \ref{Dinew1} we see $\tF^{nn} \geq \theta \sum \tF^{ii} \geq \theta p$.
We then obtain that
\[
\tF^{nn} h_{nn}^2 \geq \delta^2 \kappa_1^2 \tF^{nn} \geq \theta p \delta^2 \kappa_1^2.
\]
Therefore,
for sufficiently large $\kappa_1$, we have
\[
Ch_{11} \leq \frac{ a \tF^{nn}h_{nn}^2 }{2 (u-a)}.
\]

Case 2. Now $\kappa_n \geq -\delta \kappa_1$. We further divide this case into two cases.

Subcase 2.1.
Suppose $\kappa_{n-p+1} + \kappa_{n-p+2} + \cdots + \kappa_n < \frac{\delta}{\kappa_1}$.
Since
\[
F^{nn} \geq \frac{F(\kappa)}{\kappa_{n-p+1} + \kappa_{n-p+2} + \cdots + \kappa_n},
\]
we see that
\[
\tF^{nn} = \frac{1}{C_n^p} F^{\frac{1}{C_n^p} - 1} F^{nn} \geq \frac{F^{ \frac{1}{C_n^p} } }{C_n^p}  \frac{\kappa_1}{\delta}
= \frac{ \tilde f }{C_n^p}  \frac{\kappa_1}{\delta}.
\]
Choosing $\delta$ sufficiently small, we obtain that
\[
C \kappa_1 \leq \tF^{nn}.
\]

Subcase 2.2. Suppose $\kappa_{n-p+1} + \kappa_{n-p+2} + \cdots + \kappa_n \geq \frac{\delta}{\kappa_1}$.
For a fixed $(p-1)$-tuple $2\leq i_1 < \cdots < i_{p-1} \leq n$, we have
\[
\kappa_1 + \kappa_{i_1} + \cdots + \kappa_{i_{p-1}} \geq (1-(p-1)\delta ) \kappa_1.
\]
Hence, we have
\[\begin{aligned}
F^{nn} \geq &\; \mathop{\Pi}\limits_{2\leq i_1 < \cdots < i_{p-1} \leq n } ( \kappa_1 + \kappa_{i_1} + \cdots + \kappa_{i_{p-1}} ) \\
& \times \mathop{\Pi}\limits_{ \substack{ 2\leq i_1 < \cdots < i_{p} \leq n \\ (i_1, \cdots, i_p) \neq (n-p+1, \cdots, n) } }
( \kappa_{i_1} + \kappa_{i_2} + \cdots + \kappa_{i_{p}} )\\
\geq &\; [ (1-(p-1)\delta ) \kappa_1 ]^{C_{n-1}^{p-1}} [\frac{\delta}{\kappa_1}]^{C_{n-1}^p - 1}.
\end{aligned}\]
For $p\geq \frac{n}{2}$, a direct calculation shows that
\[
C_{n-1}^{p-1} - C_{n-1}^{p} = \frac{(n-1)\cdots(n-p+1)}{(p-1)!} \Big( 1 - \frac{n-p}{p} \Big) \geq 0.
\]
Therefore, we obtain
\[
F^{nn} \geq c_{\delta} \kappa_1,
\]
where $c_\delta = [ (1-(p-1)\delta ) ]^{C_{n-1}^{p-1}} \delta^{C_{n-1}^p - 1}$.
It then follows that, for sufficiently large $A$,
\[
C\kappa_1 \leq \frac{A}{2} \tF^{nn}.
\]

\end{proof}

By the above Lemma, \eqref{L1} becomes
\begin{equation}
\label{S2-5}
\begin{aligned}
0 \geq &\,   - \frac{2}{h_{11}}  \sum_{i\geq 2} \tF^{1i,i1} h_{11i}^2 - \frac{\tF^{ii}h_{11i}^2}{h_{11}^2} + \frac{ a \tF^{ii}h_{ii}^2 }{2(u-a)} \\
&\; + \frac{ \tF^{ii}u_i^2 }{(u-a)^2}
+ \frac{A}{2} \sum \tF^{ii} - CA.
\end{aligned}
\end{equation}

\begin{lemma}
\label{lem3}
For $\kappa_1$ sufficiently large, we have
\[
|\kappa_{n-p+1}|, \cdots, |\kappa_n| \leq C A.
\]
\end{lemma}

\begin{proof}
By the critical equation \eqref{diff1} and the Cauchy-Schwarz inequality, we have
\begin{equation}
\label{CS}
- \frac{\tF^{ii}h_{11i}^2}{h_{11}^2} \geq - (1+ \varepsilon ) \frac{ \tF^{ii}u_i^2 }{(u-a)^2}
- \Big( 1 + \frac{1}{\varepsilon}\Big) A^2 \tF^{ii} \langle X, e_i\rangle^2.
\end{equation}
From \eqref{S2-5} and $- \tF^{1i,i1} \geq 0$, we see that
\begin{equation}
\label{S2-6}
\begin{aligned}
0 \geq &\;   \frac{ a \tF^{ii}h_{ii}^2 }{2(u-a)} - \frac{\varepsilon \tF^{ii}u_i^2 }{(u-a)^2} - \frac{CA^2}{\varepsilon} \sum \tF^{ii} - CA.
\end{aligned}
\end{equation}
Using $u_i = h_{ii}\langle X, e_i\rangle$ and choosing $\varepsilon$ sufficiently small, we obtain from \eqref{S2-6} that
\begin{equation}
\label{S2-7}
\begin{aligned}
0 \geq &\;   \frac{ a \tF^{ii}h_{ii}^2 }{4(u-a)} - \frac{CA^2}{\varepsilon} \sum \tF^{ii},
\end{aligned}
\end{equation}
where we also used $\sum \tF^{ii} \geq p$. By Lemma \ref{Dinew1},
we now arrive at
\begin{equation}
\label{S2-8}
\begin{aligned}
0 \geq &\;   \frac{ a \theta  }{4(u-a)} \Big( \sum \tF^{ii} \Big) \Big( \sum_{i\geq n-p+1} h_{ii}^2 \Big) - \frac{CA^2}{\varepsilon} \sum \tF^{ii},
\end{aligned}
\end{equation}
which implies that
\[
\sum_{i\geq n-p+1} h_{ii}^2 \leq CA^2.
\]

\end{proof}


\begin{lemma}
\label{lem4}
Given $1> \delta > 0$, there is an $\epsilon = \epsilon (p, \delta) > 0$
such that,
\[
-2 F^{1i,i1} + 2 \frac{F^{11}}{\kappa_1} \geq (1+ \delta) \frac{F^{ii}}{\kappa_1}, \; i= 2,3, \cdots n,
\]
for $\kappa_1$ sufficiently large.
\end{lemma}

\begin{proof}
Recall that $\kappa_1 \geq \kappa_2 \geq \cdots \geq \kappa_n$.
By the formula $F^{1i,i1} = \frac{F^{11} -F^{ii}}{\kappa_1 -\kappa_i}$, we see that
\[
\frac{F^{ii}}{\kappa_1} = \frac{\kappa_1 -\kappa_i}{\kappa_1}(- F^{1i,i1}) + \frac{F^{11}}{\kappa_1}.
\]
Since $\kappa_{i} > 0$ for $ i \leq n-p+1$, we obtain
\[
\frac{F^{ii}}{\kappa_1} \leq - F^{1i,i1} + \frac{F^{11}}{\kappa_1}, \; \mbox{for}\; i = 2, \cdots, n-p+1.
\]
By Lemma \ref{lem3}, we can assume that $|\kappa_i| \leq \epsilon \kappa_1$ for $i \geq n-p+2$
for sufficiently small $\epsilon$ and large $\kappa_1$.
Hence, we have
\[
\frac{F^{ii}}{\kappa_1} \leq - (1+\epsilon) F^{1i,i1} + \frac{F^{11}}{\kappa_1}, \; \mbox{for}\; i = n-p+2, \cdots, n.
\]
By the above two inequalities, we get the desired inequality.
\end{proof}

By Lemma \ref{lem4}, \eqref{S2-5} becomes

\begin{equation}
\label{S2-5'}
\begin{aligned}
0 \geq &\,  \sum_{i\geq 2}\frac{\tF^{ii}h_{11i}^2}{h_{11}^2} - \frac{\tF^{ii}h_{11i}^2}{h_{11}^2}
- 2 \sum_{i\geq 2} \frac{\tF^{11}h_{11i}^2}{h_{11}^2}
+ \frac{ a \tF^{ii}h_{ii}^2 }{2(u-a)} \\
&\; + \frac{ \tF^{ii}u_i^2 }{(u-a)^2}
+ \frac{A}{2} \sum \tF^{ii} - CA.
\end{aligned}
\end{equation}
By the critical equation \eqref{diff1}, the Cauchy-Schwarz inequality and $u_i = h_{ii}\langle X, e_i\rangle$, we see that
\begin{equation}
\label{h111}
\begin{aligned}
- \frac{\tF^{11}h_{111}^2}{h_{11}^2} \geq &\; - (1+ \varepsilon ) \frac{ \tF^{11}u_1^2 }{(u-a)^2}
- \Big( 1 + \frac{1}{\varepsilon}\Big) A^2 \tF^{11} \langle X, e_1\rangle^2\\
\geq &\; - \frac{ \tF^{11}u_1^2 }{(u-a)^2} - C\varepsilon \frac{ \tF^{11}h_{11}^2 }{(u-a)^2} - \frac{CA^2 }{\varepsilon} \tF^{11}
\end{aligned}
\end{equation}
and
\[
2 \sum_{i\geq 2} \frac{\tF^{11}h_{11i}^2}{h_{11}^2} \leq C \sum_{i\geq 2} \tF^{11} h_{ii}^2 + CA^2 F^{11}.
\]
Note that
\[
F^{11} = \sum_{1\notin \{i_1, \cdots, i_{p-1}\}} \frac{F(\kappa)}{\kappa_1 + \kappa_{i_1} + \cdots + \kappa_{i_{p-1}}}
\leq \frac{C}{\kappa_1}.
\]
Hence, by Lemma \ref{lem2}, we obtain
\begin{equation}
\label{h11i}
2 \sum_{i\geq 2} \frac{\tF^{11}h_{11i}^2}{h_{11}^2} \leq \frac{ a \tF^{ii}h_{ii}^2 }{4 (u-a)} + \frac{A}{4} \sum \tF^{ii} + CA^2 \tF^{11}.
\end{equation}
Substituting \eqref{h111} and \eqref{h11i} into \eqref{S2-5'},
we have
\begin{equation}
\label{S2-5''}
\begin{aligned}
0 \geq &\,  \frac{ a \tF^{ii}h_{ii}^2 }{4(u-a)} - \frac{ \tF^{11}u_1^2 }{(u-a)^2}
- C\varepsilon \frac{ \tF^{11}h_{11}^2 }{(u-a)^2} - \frac{CA^2 }{\varepsilon} \tF^{11}\\
 + &\; \frac{ \tF^{ii}u_i^2 }{(u-a)^2}
+ \frac{A}{4} \sum \tF^{ii} - CA^2 \tF^{11} - CA.
\end{aligned}
\end{equation}
Choosing $\varepsilon$ sufficiently small and assuming $h_{11}$ sufficiently large, we derive that
\begin{equation}
\label{S2-9}
\begin{aligned}
0 \geq &\,  \frac{ a \tF^{ii}h_{ii}^2 }{8(u-a)}
+ \frac{A}{4} \sum \tF^{ii} - CA.
\end{aligned}
\end{equation}
It then follows that
\[
\sum \tF^{ii} \leq C.
\]
Next we prove that under this condition, one have
\[
\tF^{11} \geq \frac{1}{C}.
\]

Since $\sum F^{ii} \leq C$, in particular we have
\[
\frac{F(\kappa)}{\kappa_{n-p+1} + \cdots + \kappa_{n}} \leq C.
\]
This implies that
\[
\kappa_{n-p+1} + \cdots + \kappa_{n} \geq \frac{1}{C},
\]
where $C$ also depends on $\inf f$.
This yields
\[
\tF^{ii} \geq \frac{1}{C}, \; \forall \; 1\leq i \leq n.
\]
Substituting the above inequality into \eqref{S2-9}, we obtain
\[
0\geq \frac{h_{11}^2}{C} - CA,
\]
from which we can derive an upper bound for $h_{11}$.
Theorem \ref{thm1} is proved.

\section{Gradient Estimates}

Before we apply the continuity method to obtain a solution to equation \eqref{eqn}, we need to derive a $C^1$ estimate for the equation.
We show that there exists a positive constant $C$ depending on $n$, $p$, $\inf \rho$, $\sup \rho$, $\inf f$ and $|f|_{C^1}$ such that
\[
|\nabla \rho| \leq C,
\]
where $\nabla$ denotes the connection on $\mathbb{S}^n$.
Note that
\[
u = \frac{\rho^2}{ \sqrt{\rho^2 + |\nabla \rho|^2 } }.
\]
We only need to derive a positive lower bound of $u$. As in \cite{GLM}
we consider the following quantity
\[
w = -\log u + \gamma (|X|^2),
\]
where the function $\gamma (\cdot)$ will be determined later.
Suppose the maximum of $w$ is achieved at $X_0 \in M$. If at $X_0$, $ X$ is parallel to $\nu$, we have
\[
u = \langle X, \nu \rangle = \rho \geq \inf_M \rho,
\]
which gives a lower bound since $\rho$ is assumed to have a positive lower bound.
If at $X_0$ , $X$ is not parallel to $\nu$,
we can choose a local orthonormal frame $\{e_1, e_2, \cdots, e_n\}$ such that
\[
\langle X, e_1 \rangle \neq 0 \; \mbox{and}\; \langle X, e_i \rangle = 0 \; \mbox{for}\; i \geq 2.
\]
Differentiate $w$ at $X_0$ to obtain that
\begin{equation}
\label{w-diff1}
0 = w_i = - \frac{u_i}{u} + 2 \gamma' \langle X, e_i \rangle = - \frac{h_{i1} \langle X, e_1 \rangle}{u} + 2 \gamma' \langle X, e_i \rangle,
\end{equation}
where in the last equality we used the Weingarten equation. Hence, we have
\[
h_{11} = 2 \gamma' u \; \mbox{and} \; h_{1i} = 0 \; \mbox{for} \; i \geq 2.
\]
Without loss of generality, we can assume $\{h_{ij}\}$ is diagonal at $X_0$.
Differentiating $w$ at $X_0$ a second time and contracting with $\{\tF^{ij}\}$, we obtain that
\begin{equation}
\label{w-diff2}
0 \geq \tF^{ii} \Big(- \frac{u_{ii}}{u} + \frac{u_i^2}{u^2} + \gamma'' (|X|^2)_i^2 + \gamma' (|X|^2)_{ii} \Big).
\end{equation}
Combining \eqref{w-diff1} with the above inequality, we arrive at
\begin{equation}
\label{S1-1}
0 \geq  - \frac{\tF^{ii} u_{ii}}{u} + 4(\gamma'^2 + \gamma'') F^{11} \langle X, e_1\rangle^2 + \gamma'\tF^{ii} (|X|^2)_{ii}.
\end{equation}

By \eqref{Fu}, we have
\begin{equation}
\begin{aligned}
\tF^{ii} u_{ii} = \langle X, e_1\rangle ((d_X \tilde f) (e_1) + h_{11} (d_\nu \tilde f) (e_1) ) - u \tF^{ii}h_{ii}^2 + \tilde f.
\end{aligned}
\end{equation}
Also, we have
\[
\tF^{ii} (|X|^2)_{ii} = 2 \sum \tF^{ii} - 2 u \tilde f,
\]
where we used $\tF^{ij} h_{ij} = \tilde f$.
Recall that $h_{11} = 2 \gamma' u$.
Substituting the above two equalities into \eqref{S1-1} we get
\begin{equation}
\label{S1-2}
\begin{aligned}
0 \geq &\; - \frac{1}{u} \Big( \langle X, e_1\rangle (d_X \tilde f) (e_1) + \tilde f \Big)
- 2\langle X, e_1\rangle \gamma' (d_\nu \tilde f) (e_1)  \\
&\; + \tF^{ii}h_{ii}^2 + 4(\gamma'^2 + \gamma'') F^{11} \langle X, e_1\rangle^2 + 2 \gamma' ( \sum \tF^{ii} - u \tilde f ).
\end{aligned}
\end{equation}
At $X_0$, we see that $X = \langle X, e_1 \rangle e_1 + \langle X, \nu \rangle \nu$. It then follows that
\[
(d_X \tilde f) (X) = \langle X, e_1 \rangle (d_X \tilde f) (e_1) +  \langle X, \nu \rangle (d_X \tilde f) (\nu).
\]
From \eqref{f2}, we see that
\[\begin{aligned}
0\geq &\; \frac{\partial }{\partial \rho} \Big( \rho^{C_n^p} f(X, \nu) \Big)
= \frac{\partial }{\partial \rho} \Big( \rho^{C_n^p} \tilde f^{C_n^p} (X, \nu) \Big)\\
= &\; {C_n^p}(\rho \tilde f)^{{C_n^p}-1} \Big( \tilde f + (d_X \tilde f)(X) \Big)\\
= &\; {C_n^p}(\rho \tilde f)^{{C_n^p}-1} \Big( \tilde f +
\langle X, e_1 \rangle (d_X \tilde f) (e_1) +  \langle X, \nu \rangle (d_X \tilde f) (\nu) \Big).
\end{aligned}\]
We therefore obtain
\[
- ( \tilde f + \langle X, e_1 \rangle (d_X \tilde f) (e_1) ) \geq \langle X, \nu \rangle (d_X \tilde f) (\nu) = u (d_X \tilde f) (\nu).
\]
Substituting this into \eqref{S1-2} we obtain
\begin{equation}
\label{S1-3}
\begin{aligned}
0 \geq &\; (d_X \tilde f) (\nu)
- 2\langle X, e_1\rangle \gamma' (d_\nu \tilde f) (e_1) + \tF^{ii}h_{ii}^2 \\
 &\; + 4(\gamma'^2 + \gamma'') F^{11} \langle X, e_1\rangle^2 + 2 \gamma' ( \sum \tF^{ii} - u \tilde f ).
\end{aligned}
\end{equation}

Now we choose $\gamma (t) = \frac{\alpha}{t}$, where $\alpha$ is a large constant to be determined later. Recall that
$h_{11} = 2\gamma' u$ at $X_0$, which implies that $h_{11} (X_0) < 0$.
This means that $h_{11} \in \{\kappa_{n-p+2}, \kappa_{n-p+3}, \cdots, \kappa_n \}$ and therefore by Lemma \ref{Dinew1}
\[
F^{11} \geq \theta \sum F^{ii}.
\]
Similar to \cite{CJ}, we can assume $\langle X, e_1\rangle^2 \geq \frac{1}{2} \inf_M \rho^2$.
Now we arrive at
\begin{equation}
\label{S1-4}
\begin{aligned}
0 \geq &\; \Big(\frac{\alpha^2}{C } - C\alpha \Big)\sum \tF^{ii} - C\alpha.
\end{aligned}
\end{equation}
Choosing $\alpha$ sufficiently large, we obtain a contradiction.
Therefore, $X_0$ is parallel to $\nu$, and $u$ has a positive lower bound.

\section{Existence of a solution}

We use the continuity method as in \cite{CNS4} to prove Theorem \ref{thm2}.
Consider the following family of functions
\[
f^t(X,\nu) = t f(X,\nu) + (1-t) p^{C_n^p} [ \frac{1}{|X|^{C_n^p}} + \varepsilon (\frac{1}{|X|^{C_n^p}} - 1) ],
\]
where $\varepsilon$ is a small positive constant such that
\[
\min_{r_1\leq \rho\leq r_2} [ \frac{1}{\rho^{C_n^p}} + \varepsilon (\frac{1}{\rho^{C_n^p}} - 1) ] \geq c_0 > 0,
\]
for some positive constant $c_0$.
It is easy to see that$f^t (X,\nu)$ satisfies \eqref{f1} and \eqref{f2} with strict inequalities for $0\leq t < 1$.

Let $M_t$ be the solution of the equation
\[
F (\kappa) = f^t (X_t, \nu_t),
\]
where $X_t$ and $\nu_t$ are position vector and unit outer normal of $M_t$ respectively.
Clearly, when $t=0$, we have $M_0 = \mathbb{S}^n$ and $X_0 = x$. For $t\in (0,1)$, suppose
$\rho_t = |X_t|$ attains its maximum at the point $x_0$. At this point, by \eqref{secondForm}, we have
\[
g_{ij} = \rho_t^2 \delta_{ij}\; \mbox{and}\; h_{ij} = - (\rho_t)_{ij} + \rho_t \delta_{ij} \geq \rho_t \delta_{ij}
\]
under a smooth local orthonormal frame on $\mathbb{S}^n$.
Then, we have
\[
F (\kappa) \geq F (\frac{1}{\rho_t} (1, \cdots, 1) ) = \frac{p^{C_n^p}}{\rho_t^{C_n^p}}.
\]
On the other hand, at $x_0$, the unit outer normal $\nu_t$ is parallel to $X_t$.
If $\rho_t (x_0) = r_2$, we obtain
\[
\frac{p^{C_n^p}}{{r_2}^{C_n^p}} \leq F (\kappa) = f^t (X_t, \nu_t) < \frac{p^{C_n^p}}{{r_2}^{C_n^p}},
\]
which is a contradiction. So we have $\sup_{M_t} \rho_t \leq r_2$. Similarly argument at the minimum point of $\rho_t$ gives
that $\inf_{M_t} \rho_t \geq r_1$ on $M_t$. Hence, $C^0$ estimate follows. Combining our $C^1$ estimate, our $C^2$ estimate,
the Evans-Krylov theorem with the argument in \cite{CNS4}, we get the existence and uniqueness of solution to equation \eqref{eqn}.
Theorem \ref{thm2} is proved.

\section{Proof of Theorem \ref{thm3}}

By Lemmas \ref{lem1}, \ref{lem2}, \ref{lem3} and \ref{lem4} one can prove Theorem \ref{thm3}.
For completeness, we include an outline here.

\begin{proof}
We consider the following function
\[
G (x, \xi) = \log u_{\xi\xi} + \frac{a}{2} |\nabla u|^2 + \frac{A}{2} |x|^2 + \beta \log (v-u),
\]
where $a, A$, and $\beta$ are constants to be determined later.
Suppose that $G$ achieves its maximum at $(x_0, \xi_0)$.
Around $x_0$, we a choose coordinate system such that $\xi_0 = e_1$ and
$u_{ij} (x_0)$ is diagonal such that
\[
u_{11} \geq u_{22} \geq \cdots \geq u_{nn} \; \mbox{at} \; x_0.
\]
This can be done as in \cite{G}.
Thus, the new function defined by
\[
\hat G (x) =  \log u_{11} + \frac{a}{2} |\nabla u|^2 + \frac{A}{2} |x|^2 + \beta \log (v-u)
\]
also attains its maximum at $x_0$.
Differentiate it once to obtain
\begin{equation}
\label{critical}
0 = \frac{u_{11i}}{u_{11}} + a u_i u_{ii} + A x_i + \frac{\beta (v-u)_i}{v-u}.
\end{equation}
Differentiating it twice
and by similar computations as Lemma \ref{lem1} and Lemma \ref{lem2}
we can arrive at
\begin{equation}
\label{In2-1}
\begin{aligned}
0 \geq &\,   - \frac{2}{u_{11}}  \sum_{i\geq 2} \tF^{1i,i1} u_{11i}^2 - \frac{\tF^{ii} u_{11i}^2}{u_{11}^2} + \frac{ a \tF^{ii}u_{ii}^2 }{2} \\
&\; + \frac{A}{2} \sum \tF^{ii} - CA - \frac{\beta \tF^{ii} (v - u)_i^2}{(v-u)^2} - \frac{C\beta}{v-u}.
\end{aligned}
\end{equation}
We remark that in the above inequality we used
\[\begin{aligned}
\sum_k F^{kk} v_{kk} = &\; \sum_k \sum_{k\in\{i_1, \cdots, i_p\}}  \frac{F(\lambda)}{\lambda_{i_1} + \cdots + \lambda_{i_p}} v_{kk} \\
= &\; \sum_{1\leq i_1 < i_2 < \cdots < i_p \leq n} \frac{F(\lambda)}{\lambda_{i_1} + \cdots + \lambda_{i_p}} (v_{i_1i_1} + \cdots + v_{i_p i_p})\\
\end{aligned}\]
which is nonnegative since $v$ is $p$-plurisubharmonic.
Using the same argument as in Lemma \ref{lem3}, we obtain
\[
|u_{ii}| \leq \frac{C}{v-u} \; \mbox{for}\; i \geq n-p+1,
\]
where $C$ depends on $a, A$ and $\beta$.

By \eqref{critical} and the Cauchy-Schwarz inequality, we have
\[
-\frac{\tF^{11} u_{111}^2 }{u_{11}^2} \geq - Ca^2 \tF^{11} u_{11}^2 - CA^2 \tF^{11} - \frac{C\beta^2 \tF^{11}}{(v-u)^2}
\]
and
\[
- \sum_{i\geq 2} \frac{\beta \tF^{ii} (v-u)_i^2}{( v-u)^2} \geq - \frac{3}{\beta} \sum_{i\geq 2} \frac{\tF^{ii} u_{11i}^2}{u_{11}^2}
-\frac{Ca^2}{\beta} \sum_{i\geq 2} \tF^{ii} u_{ii}^2 - \frac{CA^2}{\beta} \sum_{i\geq 2} \tF^{ii}.
\]
Substituting the above two inequalities into \eqref{In2-1} we have
\begin{equation}
\label{In2-2}
\begin{aligned}
0 \geq &\,   - \frac{2}{u_{11}}  \sum_{i\geq 2} \tF^{1i,i1} u_{11i}^2
- \Big(1+\frac{3}{\beta}\Big) \sum_{i\geq 2}\frac{\tF^{ii} u_{11i}^2}{u_{11}^2}
+ \Big( \frac{ a  }{2} -\frac{Ca^2}{\beta} \Big) \tF^{ii}u_{ii}^2\\
&\; + \Big( \frac{A}{2} - \frac{CA^2}{\beta} \Big) \sum \tF^{ii} - Ca^2 \tF^{11} u_{11}^2 - CA^2 \tF^{11} - \frac{C\beta^2 \tF^{11}}{(v-u)^2} \\
&\; - CA - \frac{2 \beta \tF^{11} u_1^2}{(v-u)^2} - \frac{C\beta}{v-u}.
\end{aligned}
\end{equation}
By Lemma \ref{lem2}, similar to \eqref{h11i}, we can get
\begin{equation}
\label{u11i}
2 \sum_{i\geq 2} \tF^{11} \frac{u_{11i}^2}{u_{11}^2} \leq \frac{a}{8} \tF^{ii} u_{ii}^2 + \frac{A}{8} \sum \tF^{ii} +
CA^2 \tF^{11} + \frac{C\beta^2}{(v-u)^2}\tF^{11}
\end{equation}
for sufficiently large $(v-u)u_{11}$ and $A$.
Combining Lemma \ref{lem4} with \eqref{u11i} and choosing $a$ sufficiently small and $\beta$ sufficiently large such that $\delta \geq \frac{3}{\beta}$,
we get from \eqref{In2-2} that
\begin{equation}
\label{In2-4'}
\begin{aligned}
0 \geq &\, \frac{ a  }{4} \tF^{11}u_{11}^2  - \frac{C }{(v-u)^2} \tF^{11}  - CA  - \frac{C}{v-u}\\
\geq &\, \frac{ a  }{8}   \tF^{11}u_{11}^2 - \frac{C}{v-u},
\end{aligned}
\end{equation}
where in the second inequality we assumed $(v-u)u_{11}$ is large enough.

By Lemma \ref{Dinew1} (1) we have that $\tF^{11} u_{11} \geq c_0$, where $c_0 > 0$ depends on $\inf \tilde f$.
From \eqref{In2-4'}, we then obtain
\[
(v-u) u_{11} \leq C
\]
which implies the estimate \eqref{esti-2}.

\end{proof}

\section{Appendix}

\appendix
\renewcommand{\appendixname}{Appendix~\Alph{section}}

In this appendix, we include a proof of Lemma 2.2.
For $A\in P_p$, recall the notations
\[
F (A) := F (\lambda(A)) = \Pi_{1\leq i_1 < \cdots < i_p\leq n} (\lambda_{i_1} + \cdots + \lambda_{i_p}) \; \mbox{and} \; \tilde F = F^{\frac{1}{C_n^p}},
\]
where $\lambda(A) = (\lambda_1, \cdots, \lambda_n)$ are the eigenvalues of $A$ and $P_p$ is defined in Definition \ref{def}.
Suppose that the diagonal matrix $A = diag(\lambda_1, \cdots, \lambda_n)$ belongs to $P_p$
and $\lambda_1 \geq \cdots \geq \lambda_n$.

\begin{lemma}
$\tF^{11} (A) \lambda_1 \geq \frac{1}{n} \tF(A)$
\end{lemma}

\begin{proof}
We have that
\[\begin{aligned}
\tF^{11} (A)  =&\; \frac{1}{C_n^p} [F(A)]^{\frac{1}{C_n^p} -1} \sum_{1\in \{i_1, \cdots, i_p\}}
\frac{F(A)}{\lambda_{i_1} + \cdots + \lambda_{i_p}}\\
\geq &\; \frac{1}{C_n^p} [F(A)]^{\frac{1}{C_n^p} -1} C_{n-1}^{p-1}
\frac{F(A)}{p\lambda_1}\\
= &\; \frac{1}{n\lambda_1} [F(A)]^{\frac{1}{C_n^p}  },
\end{aligned}\]
where in the inequality we used $\lambda_{i_1 } + \cdots + \lambda_{i_p} \leq p \lambda_1$.
\end{proof}

\begin{lemma}
$\sum_{k=1}^n \tF^{kk} (A) \geq p$.
\end{lemma}

\begin{proof}
We have that
\[\begin{aligned}
\sum_{k=1}^n \tF^{kk} (A) 
= &\; \frac{1}{C_n^p} [F(A)]^{\frac{1}{C_n^p} -1} \sum_{k=1}^{n} \sum_{k\in \{i_1, \cdots, i_p\}}
\frac{F(A)}{\lambda_{i_1} + \cdots + \lambda_{i_p}}\\
= &\; \frac{nC_{n-1}^{p-1}}{[C_n^p]^2} [F(A)]^{\frac{1}{C_n^p}}  \sum_{ 1\leq i_1 < \cdots < i_p \leq n}
\frac{1}{\lambda_{i_1} + \cdots + \lambda_{i_p}}\\
\geq &\; \frac{nC_{n-1}^{p-1}}{[C_n^p]^2} [F(A)]^{\frac{1}{C_n^p}}  \frac{C_n^p}{[F(A)]^{\frac{1}{C_n^p}}} = p,
\end{aligned}\]
where the inequality of arithmetic and geometric means was used in the inequality.
\end{proof}

\begin{lemma}
 $\sum_{k=1}^n F^{kk} (A) \lambda_{k} = C_n^p F(A)$. 
\end{lemma}

\begin{proof}
Observe that
\[\begin{aligned}
\sum_{k=1}^n F^{kk} (A) \lambda_{k}
= &\; F(A) \sum_{k=1}^{n} \sum_{k\in \{i_1, \cdots, i_p\}}
\frac{\lambda_k}{\lambda_{i_1} + \cdots + \lambda_{i_p}}\\
= &\; F(A) \frac{n C_{n-1}^{p-1}}{p} = C_n^p F(A).
\end{aligned}\]
\end{proof}

\begin{lemma}
$\sum_{k=1}^{n} \tF^{kk}(A) \lambda_k = \tF(A).$
\end{lemma}
\begin{proof}
Observe that
\[
\sum_{k=1}^n \tF^{kk} (A) \lambda_k
= \frac{1}{C_n^p} [F(A)]^{\frac{1}{C_n^p} -1} \sum_{k=1}^{n} F^{kk} (A) \lambda_k = \tF(A).
\]
\end{proof}

\begin{lemma}
There is a constant $\theta = \theta(n, p)$ such that, for all $j \geq n-p+1$,
\[
F^{jj} (A) \geq \theta \sum_{i=1}^n F^{ii}(A).
\]
\end{lemma}

\begin{proof}
Note that, for $j\geq n-p+1$,
\[
F^{jj} (A) \geq \frac{F(A)}{\lambda_{n-p+1} + \cdots + \lambda_n}
\geq \frac{1}{C_{n}^p} \sum_{ 1\leq i_1 < \cdots < i_p \leq n } \frac{F(A)}{\lambda_{i_1} + \cdots + \lambda_{i_p}},
\]
and
\[\begin{aligned}
\sum_{k=1}^n F^{kk}(A) =&\; \sum_{k=1}^{n} \sum_{k\in \{i_1, \cdots, i_p\}}
\frac{F(A)}{\lambda_{i_1} + \cdots + \lambda_{i_p}}\\
= &\; \frac{ nC_{n-1}^{p-1} }{C_n^p} \sum_{ 1\leq i_1 < \cdots < i_p \leq n}
\frac{F(A)}{\lambda_{i_1} + \cdots + \lambda_{i_p}}.
\end{aligned}\]
Thus the desired inequality is proved.
\end{proof}

\end{document}